\documentclass[12pt]{amsart}
\usepackage[utf8]{inputenc}
\usepackage{mathrsfs}
\usepackage{amssymb}
\usepackage{bm}
\usepackage{graphicx}
\usepackage[centertags]{amsmath}
\usepackage{amsfonts}
\usepackage{amsthm}
\usepackage{enumerate}
\usepackage{tocvsec2}
\usepackage{xcolor} 
\usepackage[margin=1in]{geometry}

\newtheorem{theorem}{Theorem}
\newtheorem*{theorem*}{Theorem}

\newtheorem{corollary}[theorem]{Corollary}
\newtheorem{lemma}[theorem]{Lemma}
\newtheorem{rem}[theorem]{Remark}

\newtheorem{proposition}[theorem]{Proposition}

\newtheorem{claim}[theorem]{Claim}
\newtheorem{sublemma}[theorem]{Sublemma}

\newcommand{\rr}{\mathbb{R}}
\newcommand{\ee}{\varepsilon}
\newcommand{\nn}{\mathbb{N}}
\newcommand{\supp}{\mathrm{supp}}
\newcommand{\cat}{^\smallfrown}

\begin{document}

\title{The Szlenk index of $L_p(X)$ and $A_p$}
\author{Ryan M. Causey}

\begin{abstract}Given a Banach space $X$, a $w^*$-compact subset of $X^*$, and $1<p<\infty$, we provide an optimal relationship between the Szlenk index of $K$ and the Szlenk index of an associated subset of $L_p(X)^*$.  As an application, given a Banach space $X$, we prove an optimal estimate of the Szlenk index of $L_p(X)$ in terms of the Szlenk index of $X$. This extends a result of H\'{a}jek and Schlumprecht to uncountable ordinals.   More generally, given an operator $A:X\to Y$, we provide an estimate of the Szlenk index of the ``pointwise $A$'' operator $A_p:L_p(X)\to L_p(Y)$ in terms of the Szlenk index of $A$. \end{abstract}

\maketitle

\section{Introduction}

Throughout this work, $X$ will be a fixed Banach space and $K\subset X^*$ will be a $w^*$-compact, non-empty subset.   For $1<p<\infty$, we let $K_p$ denote the $w^*$-closure in $L_p(X)^*$ of all functions of the form $gh\in L_q(X^*)\subset L_p(X)^*$, where $g:[0,1]\to K$ is simple and  Lebesgue measurable, and $h\in B_{L_q}$.  Recall that these functions act on $L_p(X)$ by $\langle gh, f\rangle = \int_0^1 \langle g(\varpi), f(\varpi)\rangle h(\varpi)d\varpi$ for $f\in L_p(X)$.  Note that if $R\geqslant 0$ is such that  $K\subset RB_{X^*}$, $K_p\subset RB_{L_p(X)^*}$, so that $K_p$ is also $w^*$-compact.    If $K=B_{X^*}$, $K_p=B_{L_p(X)^*}$ by the Hahn-Banach theorem. If $A:X\to Y$ is an operator, then there exists a ``pointwise $A$'' operator $A_p:L_p(X)\to L_p(Y)$ given by $(A_p f)(\varpi)=A(f(\varpi))$ for all $\varpi\in [0,1]$.  Then if $K=A^*B_{Y^*}$, $K_p=(A_p)^*B_{L_p(Y)^*}$, which follows from the Hahn-Banach theorem.  Thus it is natural to examine what relationship exists between $K$ and $K_p$.   In particular, one may ask what relationship exists between the Szlenk indices of these sets.  To that end, we obtain the optimal relationship.  In what follows, $\omega$ denotes the first infinite ordinal.

\begin{theorem}  Fix $1<p<\infty$.  Suppose that $\xi$ is an ordinal such that $Sz(K)\leqslant \omega^\xi$.  Then $Sz(K_p)\leqslant \omega^{1+\xi}$.    If $K$ is convex, $Sz(K_p)\leqslant \omega Sz(K)$.  If $K$ is convex and $Sz(K)\geqslant \omega^\omega$,  $Sz(K)=Sz(K_p)$.

\label{main theorem}
\end{theorem}

Using the facts stated in the introduction that $K_p=(A_p)^* B_{L_p(Y)^*}$ if $K=A^*B_{Y^*}$, we immediately deduce the following from Theorem \ref{main theorem}. 

\begin{corollary} Fix $1<p<\infty$. If $A:X\to Y$ is an operator and $K=A^*B_{Y^*}$, then $Sz(A_p)\leqslant \omega Sz(A)$, and if $Sz(A)\geqslant \omega^\omega$, $Sz(A_p)=Sz(A)$.    In particular, $A_p$ is Asplund if and only if $A$ is. 

\label{main corollary}
\end{corollary}

Applying Corollary \ref{main corollary} to the identity of a Banach space, we extend the result of H\'{a}jek and Schlumprecht from \cite{HS} to uncountable ordinals.

We recall that $K$ is said to be $w^*$-\emph{fragmentable} if for any non-empty subset $L$ of $K$ and any $\ee>0$, there exists a $w^*$-open subset $U$ of $X^*$ such that $L\cap U\neq \varnothing$ and $\text{diam}(L\cap U)<\ee$.  We recall that $K$ is $w^*$-\emph{dentable} if for any non-empty subset $L$ of $K$ and any $\ee>0$, there exists a $w^*$-open slice $S$ of $X^*$ such that $L\cap U\neq \varnothing$ and $\text{diam}(L\cap S)<\ee$.  We recall that a $w^*$-open slice is a subset of $X^*$ of the form $\{x^*\in X^*: \text{Re\ }x^*(x)>a\}$ for some $x\in X$ and $a\in \rr$. As mentioned in \cite{C2}, a consequence of Corollary \ref{main corollary} is that if $Sz(K)\leqslant \omega^\xi$, then $Sz(K)\leqslant Dz(K)\leqslant \omega^{1+\xi}$, where $Dz(K)$ denotes the $w^*$-dentability index of $K$.  Thus Corollary \ref{main corollary} implies that $K$ is $w^*$-dentable if and only if it is $w^*$-fragmentable.  

In addition to considering the Szlenk index of a set, one may consider the $\xi$-Szlenk power type $\textbf{p}_\xi(L)$ of the set $L$, which is important in $\xi$-asymptotically uniformly smooth renormings of Banach spaces and operators. The concept of a $\xi$-asymptotically uniformly smooth operator was introduced in \cite{CD}, and further sharp renorming results regarding the $\xi$-Szlenk power type of an operator were established in \cite{power}.  To that end, we have the following.

\begin{theorem} For any ordinal $\xi$ and any $1<p<\infty$, if $1/p+1/q=1$, $\textbf{\emph{p}}_{1+\xi}(K_p)\leqslant \max\{q, \textbf{\emph{p}}_\xi(K)\}$.  
\label{power type}

\end{theorem} 

In the case that $\xi\geqslant \omega$ and $\textbf{p}_\xi(K)\leqslant p$, $\textbf{p}_\xi(K)=\textbf{p}_\xi(K_p)$, in showing that Theorem \ref{power type} is sharp in some cases.

The author wishes to thank P.A.H. Brooker, P. H\'{a}jek, N. Holt, and Th. Schlumprecht for helpful remarks during the preparation of this work.

\section{$L_p(X)$, Trees, Szlenk index, games}

\subsection{Trees, $\Gamma_{\xi,n}$, $\mathbb{P}_{\xi,n}$, and stablization results} Given a set $\Lambda$, we let $\Lambda^{<\nn}$ denote the finite, non-empty sequences in $\Lambda$. Given two members $s,t$ of $\Lambda^{<\nn}$, we let $s\cat t$ denote the concatenation of $s$ and $t$, $|s|$ denotes the length of $s$, $s\preceq t$ means $s$ is an initial segment of $t$, and $s|_i$ denotes the initial segment of $s$ having length $i$.    Given $t\in \Lambda^{<\nn}$, we let $[\preceq t]=\{s\in\Lambda^{<\nn}: s\preceq t\}$. 

Any subset $T$ of $\Lambda^{<\nn}$ which contains all non-empty initial segments of its members will be called a $B$-\emph{tree}. We define by transfinite induction the \emph{derived}$B$\emph{ trees} of $T$.  We let $MAX(T')$ denote the $\preceq$-maximal members of $T$ and $T'=T\setminus MAX(T)$. We then define $T^0=T$, $T^{\xi+1}=(T^\xi)'$, and if $\xi$ is a limit ordinal, $T^\xi=\cap_{\gamma<\xi}T^\gamma$.    We let $o(T)$ denote the smallest ordinal $\xi$ such that $T^\xi=\varnothing$, provided such an ordinal exists.  If no such ordinal exists, we write $o(T)=\infty$.  We say $T$ is \emph{well-founded} if $o(T)$ is an ordinal, and $T$ is \emph{ill-founded} if $o(T)=\infty$.   For convenience, we agree to the convention that if $\xi$ is an ordinal $\xi<\infty$, and that $\omega \infty=\infty$.   

 Given a $B$-tree $T$ and a Banach space $Y$, we let $T.Y=\{(\zeta_i, Z_i)_{i=1}^k: (\zeta_i)_{i=1}^k\in T, Z_i\in \text{codim}(Y)\}$, where $\text{codim}(Y)$ denotes the closed subspaces of $Y$ having finite codimension in $Y$.   We let $\mathcal{C}$ denote the norm compact subsets of $B_X$ and $$T.X.\mathcal{C}=\{(\zeta_i, Z_i, C_i)_{i=1}^k: (\zeta_i)_{i=1}^k, Z_i \in\text{codim}(X), C_i\in \mathcal{C}\}.$$   We note that $T.Y$ and $T.X.\mathcal{C}$ are $B$-trees.  Furthermore, for any ordinal $\gamma$, $(T.Y)^\gamma=T^\gamma.Y$ and $(T.X.\mathcal{C})^\gamma=T^\gamma.X.\mathcal{C}$.    In particular, $T.Y$ and $T.X.\mathcal{C}$ have the same order as $T$.   

Given a $B$-tree $T$, a Banach space $Y$,  and a collection $(x_t)_{t\in T.Y}\subset Y$, we say $(x_t)_{t\in T.Y}$ is \emph{normally weakly null} provided that for any $t=(\zeta_i, Z_i)_{i=1}^k\in T.Y$, $x_t\in Z_k$.    Given another $B$-tree $S$ and a function $\sigma:S.Y\to T.Y$, we say $\sigma$ is a \emph{pruning} provided that for every $s, s_1\in S.Y$ with $s\prec s_1$, $\sigma(s)\prec \sigma(s_1)$, and if $s_1=s\cat (\zeta, Z)$ and $\sigma(s_1)=t\cat (\mu, W)$ for some $t\in T.Y$, $W\leqslant Z$.    If $\sigma:S.Y\to T.Y$ is a pruning and $\tau:MAX(S.Y)\to MAX(T.Y)$ is such that for every $s\in MAX(S.Y)$, $\sigma(s)\preceq \tau(s)$, we say the pair $(\sigma,\tau)$ is an \emph{extended pruning}, and denote this by $(\sigma,\tau):S.Y\to T.Y$.

For every $\xi\in \nn$ and $n\in \nn$, a $B$-tree $\Gamma_{\xi,n}$ was defined in \cite{C} so that $o(\Gamma_{\xi,n})=\omega^\xi n$.  Furthermore, a function $\mathbb{P}_\xi:\Gamma_\xi\to [0,1]$ was defined so that for every $t\in MAX(\Gamma_\xi)$, $\sum_{s\preceq t}\mathbb{P}_\xi(s)=1$.   Furthermore, $\Gamma_{\xi+1}$ is the disjoint union of $\Gamma_{\xi,n}$, $n\in \nn$.    For convenience, we define $\mathbb{P}_{\xi,n}:\Gamma_{\xi,n}\to [0,n]$ by $\mathbb{P}_{\xi,n}(s)=n \mathbb{P}_{\xi+1}(s)$. It follows from the definitions that $\Gamma_{\xi,1}=\Gamma_\xi$ and $\mathbb{P}_{\xi,1}=\mathbb{P}_\xi$.   For every $\xi$ and every $n\in \nn$, there exist disjoint subsets $\Lambda_{\xi,n,1}$, $\ldots$, $\Lambda_{\xi,n,n}$ of $\Gamma_{\xi,n}$ such that $\Gamma_{\xi,n}=\cup_{i=1}^n \Lambda_{\xi,n,i}$.   It follows from the facts regarding $\mathbb{P}_{\xi+1}$ discussed in \cite{C} that, with these definitions, for every ordinal $\xi$, every $n\in \nn$, every $1\leqslant i\leqslant n$, and every $t\in MAX(\Gamma_{\xi,n})$, $\sum_{\Lambda_{\xi,n,i}\ni s\preceq t} \mathbb{P}_{\xi,n}(s)=1$.  For any Banach space $Y$, we may define $\mathbb{P}_{\xi,n}$ on $\Gamma_{\xi,n}.Y$ and $\Gamma_{\xi,n}.X.\mathcal{C}$ by letting $$\mathbb{P}_{\xi,n}((\zeta_i, Z_i)_{i=1}^k)=\mathbb{P}_{\xi,n}((\zeta_i)_{i=1}^k)$$ and $$\mathbb{P}_{\xi,n}((\zeta_i, Z_i, C_i)_{i=1}^k)= \mathbb{P}_{\xi,n}((\zeta_i)_{i=1}^k).$$    We say an extended pruning $(\sigma, \tau):\Gamma_{\xi,n}.X\to \Gamma_{\xi,n}.X$ is \emph{level preserving} provided that for every $1\leqslant i\leqslant n$, $\sigma(\Lambda_{\xi,n,i})\subset \Lambda_{\xi,n,i}$.

The following theorem collects results from Theorem $3.3$, Propositions $3.2$, $3.3$, and Lemma $3.4$ of \cite{power}.  

\begin{theorem} Suppose $\xi$ is an ordinal and $n$ is a natural number.  \begin{enumerate}[(i)]\item If $f:\Pi(\Gamma_\xi.n.X)\to \mathbb{R}$ is bounded and $\lambda\in \mathbb{R}$ is such that $$\lambda<\inf_{t\in MAX(\Gamma_{\xi.n}.X)} \sum_{s\preceq t}\mathbb{P}_{\xi,n}(s) f(s,t),$$ then there exist a level preserving extended pruning $(\sigma,\tau):\Gamma_{\xi,n}.X\to \Gamma_{\xi,n}.X$ and real numbers $b_1, \ldots, b_n$ such that $\lambda<\sum_{i=1}^n b_i$ and for every $1\leqslant i\leqslant n$ and every $\Lambda_{\xi, n, i}\ni s\preceq t\in MAX(\Gamma_{\xi,n}.X)$, $b_i\leqslant f(\sigma(s), \tau(t))$. \item If $(M,d)$ is a compact metric space and  $f:\Pi(\Gamma_{\xi,n}.X)\to M$ is any function, then for any $\delta>0$, there exist $x_1, \ldots, x_n\in M$ and a level preserving extended pruning $(\sigma,\tau):\Gamma_{\xi,n}.X\to \Gamma_{\xi,n}.X$ such that for every $1\leqslant i\leqslant n$ and every $\Lambda_{\xi,n,i}\ni s\preceq t\in MAX(\Gamma_{\xi,n}.X)$, $d(x_i, f(\sigma(s), \tau(t)))<\delta$. \item If $F$ is a finite set and $f:MAX(\Gamma_{\xi,n}.X)\to F$ is any function, there exists a level preserving extended pruning $(\sigma, \tau):\Gamma_{\xi,n}.X\to \Gamma_{\xi,n}.X$ such that $f\circ \tau|_{MAX(\Gamma_{\xi,n}.X)}$ is constant. \item For any natural numbers $k_1<\ldots <k_r\leqslant n$, there exists an extended pruning $(\sigma, \tau):\Gamma_{\xi,r}.X\to \Gamma_{\xi, n}.X$ such that for every $1
\leqslant i\leqslant r$,  $\sigma(\Lambda_{\xi, n,i})\subset \Lambda_{\xi, n, k_i}$. \end{enumerate}

\label{stabilization}
\end{theorem}

\subsection{The Szlenk index, Szlenk power type}

Given a $w^*$-compact subset $L$ of $X^*$ and $\ee>0$, we let $s_\ee(K)$ denote the set consisting of those $x^*\in L$ such that for every $w^*$-neighborhood $V$ of $x^*$, $\text{diam}(L\cap V)>\ee$.    We define the transfinite derivations $$s_\ee^0(L)=L,$$ $$s_\ee^{\xi+1}(L)=s_\ee(s_\ee^\xi(L)),$$ and if $\xi$ is a limit ordinal, $$s_\ee^\xi(L)= \bigcap_{\zeta<\xi}s_\ee^\zeta(L).$$   If there exists an ordinal $\xi$ such that $s_\ee^\xi(L)=\varnothing$, we let $Sz(L, \ee)$ be the minimum such ordinal.  Otherwise we write $Sz(L, \ee)=\infty$. Since $s_\ee^\xi(L)$ is $w^*$-compact, we deduce that $Sz(L, \ee)$ cannot be a limit ordinal.  We agree to the conventions that $\omega \infty=\infty$ and $\xi<\infty$ for any ordinal $\xi$.     We let $Sz(L)=\sup_{\ee>0}Sz(L, \ee)$.     If $B:Z\to W$ is an operator, we let $Sz(B, \ee)=Sz(B^*B_{W^*}, \ee)$, $Sz(B)=Sz(B^*B_{W^*})$.  If $Z$ is a Banach space, $Sz(Z, \ee)=Sz(I_Z, \ee)$ and $Sz(Z)=Sz(I_Z)$.

We recall that a set $L\subset X^*$ is called $w^*$-\emph{fragmentable} if for any $\ee>0$ and any $w^*$-compact, non-empty subset $M$ of $L$, $s_\ee(M)\subsetneq M$.  This is equivalent to $Sz(L)<\infty$.  We say an operator $B:Z\to W$ is Asplund if $B^*B_{W^*}$ is $w^*$-fragmentable, which happens if and only if $Sz(B)<\infty$.  We say a Banach space $Z$ is Asplund if $I_Z$ is Asplund.  These are not the original definitions of Asplund spaces and operators, but they are equivalent to the original definitions (see \cite{Brooker}).

If $Sz(K)\leqslant \omega^{\xi+1}$, then for any $\ee>0$, $Sz(K, \ee)\leqslant \omega^\xi n$ for some $n\in \nn$.  We let $Sz_\xi(K, \ee)$ be the smallest $n\in \nn$ such that $Sz(K, \ee)\leqslant \omega^\xi n$.    We define the $\xi$-\emph{Szlenk power type} $\textbf{p}_\xi(K)$ of $K$ by $$\textbf{p}_\xi(K)=\underset{\ee\to 0^+}{\lim\sup} \frac{\log Sz_\xi(K, \ee)}{|\log(\ee)|}.$$  This value need not be finite.  By convention, we let $\textbf{p}_\xi(K)=\infty$ if $Sz(K)>\omega^{\xi+1}$.  We let $\textbf{p}_\xi(A)=\textbf{p}_\xi(A^*B_{Y^*})$ and $\textbf{p}_\xi(X)=\textbf{p}_\xi(B_{X^*})$.   The quantities $\textbf{p}_\xi(X)$, $\textbf{p}_\xi(A)$ are important for the renorming theorem of $\xi$-asymptotically uniformly smooth norms with power type modulus.

Given a $w^*$-compact subset $L$ of $X^*$ and $\ee>0$, we let $\mathcal{H}_\ee^L$ denote the set of Cartesian products $\prod_{i=1}^n C_i$ such that $C_i\in  \mathcal{C}$ for each $1\leqslant i\leqslant n$ and such that there exist $(x_i)_{i=1}^n\in \prod_{i=1}^n C_i$ and $x^*\in K$ such that for each $1\leqslant i\leqslant n$, $\text{Re\ }x^*(x_i)\geqslant \ee$.

\subsection{The Szlenk index of $K_p$}

Recall that for $1<p<\infty$, $L_p(X)$ denotes the space of equivalence classes of Bochner integrable functions $f:[0,1]\to X$ such that $\int \|f\|^p<\infty$, where $[0,1]$ is endowed with its Lebesgue measure.  Recall also that if $1<q<\infty$, $L_q(X^*)$ is isometrically included in $L_p(X)^*$ by the action $$f\mapsto \int \langle g,f\rangle,$$ for $g\in L_q(X^*)$.    We also recall that if $\varrho:X\to \rr$ is any Lipschitz function, then for any $f\in L_p(X)$, $\varrho\circ f\in L_p$.    

We note that the Szlenk index and the $\xi$ Szlenk power type of $K$ are unchanged by scaling $K$ by a positive scalar or by replacing $K$ with its balanced hull.  Moreover, for a positive scalar $c$, $(cK)_p=cK_p$, which has the same Szlenk index and $\xi$-Szlenk power type as $K_p$.  If $\mathbb{T}K$ is the balanced hull of $K$, $K_p\subset (\mathbb{T}K)_p$ and $Sz(K)=Sz(\mathbb{T} K)$ (\cite[Lemma $2.2$]{C}) so that Theorem \ref{main theorem}, Corollary \ref{main corollary}, and Theorem \ref{power type} hold in general if they hold under the assumption that $K\subset B_{X^*}$ is balanced.  Therefore we can and do assume throughout that $K\subset B_{X^*}$ and $K$ is balanced.

Let $\varrho:X\to \rr$ be given by $\varrho(x)=\max_{x^*\in K}\text{Re\ }x^*(x)$.  Since we have assumed $K$ is balanced, $\varrho(x)=\max_{x^*\in K}|x^*(x)|$.   It is easy to see that for any $1<p<\infty$ and any $f\in L_p(X)$, $\|\varrho(f)\|_{L_p}=\max_{f^*\in K_p}\text{Re\ }f^*(f)$.   Combining this fact with \cite[Corollary $2.4$]{power} and the proof of that corollary, we obtain the following.  

\begin{theorem} Fix $1<p, \alpha<\infty$. \begin{enumerate}[(i)]\item If for every $B$-tree $T$ with $o(T)=\omega^{1+\xi}$ and every normally weakly null $(f_t)_{t\in T.L_p(X)}\subset B_{L_p(X)}$, $$\inf \bigl\{ \|\varrho(f)\|_{L_p}: t\in T.L_p(X), f\in \text{\emph{co}}(f_s: \varnothing\prec s\preceq t) \bigr\}=0,$$ then $Sz(K_p)\leqslant \omega^{1+\xi}$.     

\item If there exists a constant $C$ such that for every $n\in \nn$, every $B$-tree $T$ with $o(T)=\omega^{1+\xi} n$, and every normally weakly null collection $(f_t)_{t\in T.L_p(X)}\subset B_{L_p(X)}$, $$\inf \bigl\{\|\varrho(f)\|_{L_p}: t\in T.L_p(X), f\in \text{\emph{co}}(f_s: \varnothing\prec s\preceq t)\bigr\}\leqslant C n^{-1/\alpha}, $$    then $\textbf{\emph{p}}_{1+\xi}(K_p)\leqslant \alpha$. 

\end{enumerate}

\label{characterization}
\end{theorem}

\begin{proposition} Suppose $T$ is a non-empty $B$-tree.  Suppose also that $(C_s)_{s\in T.X}\subset \mathcal{C}$ is fixed   and for $s=(\zeta_i, Z_i)_{i=1}^k\in T.X$, let $\lambda(s)=Z_k\cap C_s$. Suppose that $S$ is a non-empty, well-founded $B$-tree and $\theta:S.X\to T.X$ is a pruning.  For $s\in S.X$, let $\mathfrak{s}(s)= \prod_{i=1}^{|s|} \lambda(\theta(s|_i))$.   If $\ee>0$ is such that for every $t\in S.X$, $\mathfrak{s}(s)\in \mathcal{H}^K_\ee\neq \varnothing$, then for any $0<\delta<\ee$, any $0\leqslant \gamma<o(S)$, and any $s\in S^\gamma.X$, $\mathfrak{s}(s)\in \mathcal{H}^{s_\delta^\gamma(K)}_\ee \neq \varnothing$.  Moreover, for any $0<\delta<\ee$, $Sz(K, \delta)>o(S)$.

\label{technical}
\end{proposition}

\begin{proof} We induct on $\gamma$. The base case is the hypothesis.    Assume $\gamma+1<o(S)$ and the result holds for $\gamma$.  Assume $s\in S^{\gamma+1}.X$, which means there exists $\zeta$ such that $s\cat (\zeta, Z)\in S^\gamma.X$ for all $Z\in \text{codim}(X)$. Then for every $Z\in \text{codim}(X)$, there exists $Z\geqslant W_Z\in \text{codim}(X)$ such that $\mathfrak{s}(s\cat(\zeta, Z))\subset \mathfrak{s}(s)\times B_{W_Z}$.  From this and the inductive hypothesis, for every $Z\in \text{codim}(X)$, we fix $x_Z\in B_{W_Z}$,  $(x^Z_i)_{i=1}^{|s|}\in \mathfrak{s}(s)$, and $x^*_Z\in s_\delta^\gamma(K)$ such that $\text{Re\ }x^*_Z(x_Z)\geqslant \ee$ and $\text{Re\ }x^*_Z(x^Z_i)\geqslant \ee$ for each $1\leqslant i\leqslant |s|$.   By compactness of $\mathfrak{s}(s)\times K$ with the product topology, where $\lambda(\theta(s|_i))$ has its norm topology and $K$ has its $w^*$-topology, $$\varnothing \neq \bigcap_{Z\in \text{codim}(X)}\overline{\{(x^Y_1, \ldots, x^Y_{|s|}, x^*_Y):Z\geqslant Y\in \text{codim}(X)\}}\subset \mathfrak{s}(s)\times K.$$  Fix $(x_1, \ldots, x_{|s|}, x^*)$ lying in this intersection.  Obviously $x^*\in s_\delta^\gamma(K)$.  Moreover, for any $w^*$-neighborhood $V$ of $x^*$, there exists $Z\in \text{codim}(X)$ such that $\ker(x^*)\subset Z$ and $x^*_Z\in V$, whence $$\text{diam}(s^\gamma_\delta(K)\cap V) \geqslant \|x^*_Z-x^*\|\geqslant \text{Re\ }(x^*_Z-x^*)(x_Z)= \text{Re\ }x^*_Z(x_Z)\geqslant \ee>\delta.$$  This implies $x^*\in s^{\gamma+1}_\delta(K)$.  It is obvious that $\text{Re\ }x^*(x_i)\geqslant \ee$ for all $1\leqslant i\leqslant |s|$.  This shows that $\mathfrak{s}(s)\in \mathcal{H}_\ee^{s^{\gamma+1}_\delta(K)}$ and completes the successor case.

Finally, assume $\gamma<o(S)$ is a limit ordinal and the result holds for all ordinals less than $\gamma$.  Fix $s\in S^\gamma.X$ and let $\mathfrak{s}(s)\times K$ be topologized as in the successor case. By the inductive hypothesis, for all $\beta<\gamma$, there exists $(x^\beta_1, \ldots, x^\beta_{|s|}, x^*_\beta)\in \mathfrak{s}(s)\times K$ such that $x^*_\beta\in s^\beta_\ee(K)$ and for all $1\leqslant i\leqslant |s|$, $\text{Re\ }x^*_\beta(x_i^\beta)\geqslant \ee$.     By compactness of $\bigl(\prod_{i=1}^{|s|} \lambda(\theta(s|_i))\bigr)\times K$, $$\bigcap_{\beta<\gamma}\overline{\{(x^\mu_1, \ldots, x^\mu_{|s|}, x^*_\mu): \mu\geqslant \beta\}}\neq \varnothing.$$   Clearly any $(x_1, \ldots, x_{|s|}, x^*)$ lying in this intersection is such that $x^*\in s^\gamma_\delta(K)$ and for any $1\leqslant i\leqslant |s|$, $\text{Re\ }x^*(x_i)\geqslant \ee$.   This shows that $\mathfrak{s}(s)\in \mathcal{H}_\ee^{s^\gamma_\delta(K)}$ and completes the induction.

We have shown that for any $0<\delta<\ee$, $Sz(K, \delta)\geqslant o(S)$.   If $o(S)$ is a limit ordinal, we deduce that $Sz(K, \delta)>o(S)$ since $Sz(K, \delta)$ cannot be a limit ordinal.     If $o(S)$ is a successor, say $o(S)=\xi+1$, then there exists a length $1$ sequence $(\zeta)\in S^\xi$.  For every $Z\in \text{codim}(X)$, $\mathfrak{s}((\zeta, Z))=W_Z\cap C_{\theta((\zeta, Z))}$ for some $W\subset Z$.  The first part of the proof yields that for each $Z\in \text{codim}(X)$, there exists $x_Z\in W_Z\cap C_{\theta((\zeta, Z))}\subset W_Z\cap B_X$ and some $x^*_Z\in s_\zeta^\xi(K)$ such that $\text{Re\ }x^*_Z(x_Z)\geqslant \ee$.    Arguing as in the successor case, we deduce that any $w^*$-limit of a subnet of $(x^*_Z)_{Z\in \text{codim}(X)}$ lies in $s^{\xi+1}_\delta(K)$, whence $Sz(K, \delta)>\xi+1=o(S)$.

\end{proof}

\subsection{Games}

Suppose $T\subset \Lambda^{<\nn}$ is a well-founded, non-empty $B$-tree and $\mathcal{E}\subset MAX(T.X.\mathcal{C})$ is some subset.  We define the \emph{game on} $T.X.\mathcal{C}$ \emph{with target set} $\mathcal{E}$.    Player I first chooses $(\zeta_1, Z_1)\in \Lambda\times \text{codim}(X)$ such that $(\zeta)\in T$ and Player II then chooses $C_1\in \mathcal{C}$.  Assuming $(\zeta_i, Z_i)_{i=1}^n\in T.X$ and $C_1, \ldots, C_n\in \mathcal{C}$ have been chosen, the game terminates if $(\zeta_i, Z_i)_{i=1}^n\in MAX(T.X)$.  Otherwise Player I chooses $(\zeta_{n+1}, Z_{n+1})\in \Lambda\times \text{codim}(X)$ such that $(\zeta_i)_{i=1}^{n+1}\in T$ and Player II chooses $C_{n+1}\in \mathcal{C}$.  Since $T$ is well-founded, this game must terminate after finitely many steps.  Suppose that the resulting choices are $(\zeta_i, Z_i)_{i=1}^n$ and $C_1, \ldots, C_n\in \mathcal{C}$.  We say that Player II wins if $(\zeta_i, Z_i, C_i)_{i=1}^n\in \mathcal{E}$, and Player I wins otherwise.  

A \emph{strategy for Player} I for the game on $T.X.\mathcal{C}$ with target set $\mathcal{E}$ is a function $\psi:T'.X.\mathcal{C}\cup \{\varnothing\}\to \Lambda\times \text{codim}(X)$ such that if $\psi((\zeta_i, Z_i, C_i)_{i=1}^{n-1})=(\zeta_n, Z_n)$, $(\zeta_i)_{i=1}^n\in T$.   We say $\psi$ is a \emph{winning strategy for Player} I provided that for any sequence $(\zeta_i, Z_i, C_i)_{i=1}^n\in MAX(T.X.\mathcal{C})$ such that $(\zeta_i, Z_i)=\psi((\zeta_j, Z_j, C_j)_{j=1}^{i-1})$ for every $1\leqslant i\leqslant n$, $(\zeta_i, Z_i, C_i)_{i=1}^n \notin \mathcal{E}$.

A \emph{strategy for Player} II for the game on $T.X.\mathcal{C}$ with target set $\mathcal{E}$ is a function $\psi$ defined on the set \begin{align*} \{((\zeta_i, Z_i, C_i)_{i=1}^{n-1}, (\zeta_n, Z_n)): & (\zeta_i, Z_i, C_i)_{i=1}^{n-1}\in \{\varnothing\}\cup T.X.\mathcal{C}, (\zeta_n, Z_n)\in \Lambda\times \text{codim}(X),  \\ & (\zeta_i)_{i=1}^n\in T\}\end{align*} and taking values in $\mathcal{C}$.   We say $\psi$ is a \emph{winning strategy for Player} II provided that for any sequence $(\zeta_i, Z_i, C_i)_{i=1}^n\in MAX(T.X.\mathcal{C})$ such that $C_i=\psi((\zeta_j, Z_j, C_j)_{j=1}^{i-1}, (\zeta_i, Z_i))$ for all $1\leqslant i\leqslant n$, $(\zeta_i, Z_i, C_i)_{i=1}^n\in \mathcal{E}$.

\begin{proposition}\cite[Proposition $3.1$]{C2} For any non-empty, well-founded $B$-tree $T$ and any $\mathcal{E}\subset T.X.\mathcal{C}$, either Player I or Player II has a winning strategy for the game on $T.X.\mathcal{C}$ with target set $\mathcal{E}$.

\end{proposition}

\begin{proposition} Suppose that Player II has a winning strategy for a game on $T.X.\mathcal{C}$ with target set $\mathcal{E}$.   Then there exists $(C_s)_{s\in T.X}\subset \mathcal{C}$ such that for every $t=(\zeta_i, Z_i)_{i=1}^k\in MAX(T.X)$, $(\zeta_i, Z_i, C_{t|_i})_{i=1}^k\in \mathcal{E}$.

\label{games}
\end{proposition}

\begin{proof} Fix a winning strategy $\psi$ for Player II in the game.   We define $C_s$ by induction on $|s|$.  We let $C_{(\zeta, Z)}=\psi(\varnothing, (\zeta,Z))$.  If $|s|=k+1$, $C_{s|_i}$ has been defined for every $1\leqslant i\leqslant k$, and $s=s|_k\cat (\zeta, Z)$, we let $C_s=\psi(s|_k, (\zeta, Z))$.

\end{proof}

For the next proposition, if $h\in L_p(X)$ is a simple function, we let $\overline{h}$ be the function in $L_p(X)$ such that $\overline{h}(\varpi)=0$ if $h(\varpi)=0$ and $\overline{h}(\varpi)=h(\varpi)/\|h(\varpi)\|$ otherwise. 

\begin{proposition} Let $\xi$ be an ordinal, $n$ a natural number,  and let $T$ be a $B$-tree with $o(T)\geqslant \omega^{1+\xi} n$.  If $\psi$ is a strategy for Player I for some game on $\Gamma_{\xi, n}.X.\mathcal{C}$, then for any $1<p<\infty$,  any $\delta>0$, and any normally weakly null $(f_t)_{t\in T.L_p(X)}\subset B_{L_p(X)}$, there exist $s=(\zeta_i, Z_i)_{i=1}^k\in MAX(\Gamma_{\xi,n}.X)$, $\varnothing=t_0\prec t_1\prec \ldots \prec t_k\in T.L_p(X)$, $g_i\in \text{\emph{co}}(f_u: t_{i-1}\prec u\preceq t_i)$, $h_i\in B_{L_p(X)}$, and $C_i\in \mathcal{C}$ such that for every $1\leqslant i\leqslant k$,  \begin{enumerate}[(i)]\item $h_i$ is simple, \item  $\text{\emph{range}}(\overline{h}_i)= C_i\subset B_{Z_i}$, \item $\|g_i-h_i\|_{L_p(X)}<\delta$, \item $(\zeta_i, Z_i)= \psi((\zeta_j, Z_j, C_j)_{j=1}^{i-1})$. 

\end{enumerate}

\label{prop8}

\end{proposition}

\begin{rem}\upshape For a $B$-tree $S$ on $\Lambda$ and $s\in S$, we let $S(s)$ denote those non-empty sequences $u\in \Lambda^{<\nn}$ such that $s\cat u\in S$.  An easy induction argument yields that for any ordinals $\xi, \zeta$,  $S^\xi(s)=(S(s))^\xi$ for any ordinal $\xi$.   From this it follows that $s\in S^\xi$ if and only if $o(S(s))\geqslant \xi$. Furthermore, another easy induction yields that if $(S^\xi)^\zeta=S^{\xi+\zeta}$, from which it follows that if $o(S)\geqslant \xi+\zeta$, $o(S^\xi)\geqslant \zeta$.  Therefore if $s\in S^{\xi+\omega}$, $o(S^\xi(s))\geqslant \omega$.

\end{rem}

\begin{proof}[Proof of Proposition \ref{prop8}] We first note that if $Z\in \text{codim}(X)$, $L_p(X)/L_p(Z)$ is either the zero vector space or isomorphic to $L_p$, and therefore has Szlenk index not exceeding $\omega$.  As explained in \cite{C2}, this means that for any $B$-tree $T$ with $o(T)\geqslant \omega$, any $\delta>0$, and any normally weakly null $(f_t)_{t\in T.L_p(X)}\subset B_{L_p(X)}$, there exist $t\in T.L_p(X)$, $g\in \text{co}(f_s: \varnothing\prec s\preceq t)$, and $h\in B_{L_p(Z)}$ such that $\|g-h\|_{L_p(X)}<\delta$.   Moreover, by the density of simple functions, we may assume this $h$ is simple. 

Let $\psi$ be a strategy for Player I for a game on $\Gamma_{\xi,n}.X.\mathcal{C}$.  Let $T$ be a $B$-tree with $o(T)=\omega^{1+\xi}n$ and define $\gamma:\Gamma_{\xi,n}.X  \cup \{\varnothing\}\to [0, \omega^\xi n]$ by letting $\gamma(t)=\max\{\mu\leqslant \omega^\xi n: t\in (\Gamma_{\xi,n}.X)^\mu\}$ for $t\in \Gamma_{\xi,n}.X$ and $\gamma(\varnothing)=\omega^\xi n$.  Let $s_0=t_0=\varnothing$.  Now assume that for some $k\in \nn$ and all $1\leqslant i<k$, $s_i\in \Gamma_{\xi,n}.X$, $\zeta_i\in [0, \omega^\xi n]$, $Z_i\in \text{codim}(X)$, $t_i\in T.L_p(X)$, $g_i, h_i\in B_{L_p(X)}$, and $C_i\in \mathcal{C}$ have been chosen such that for all $1\leqslant i<k$, \begin{enumerate}[(i)]\item $h_i$ is simple, \item $s_i=(\zeta_j, Z_j)_{j=1}^i$, \item $t_0\prec t_1\prec \ldots \prec t_{k-1}$, \item $t_i\in (T.L_p(X))^{\omega \gamma(s_i)}$, \item $(\zeta_i, Z_i)=\psi((\zeta_j, Z_j, C_j)_{j=1}^{i-1})$, \item $g_i\in \text{co}(f_u: t_{i-1}\prec u\preceq t_i)$, \item $\|g_i-h_i\|_{L_p(X)}<\delta$, \item $\text{range}(\overline{h}_i)=C_i\subset B_{Z_i}$.  \end{enumerate}    If $s_{k-1}$ is maximal in $\Gamma_{\xi,n}.X$, we let $s=s_{k-1}$, and one easily checks that the conclusions are satisfied.   Otherwise let $(\zeta_k, Z_k)=\psi((\zeta_j, Z_j, C_j)_{j=1}^{k-1})$ and $s_k=s_{k-1}\cat (\zeta_k, Z_k)$.   Let $u_{k-1}$ be the sequence of first members of the pairs of $t_{k-1}$ and let $U$  denote the proper extensions of $u_{k-1}$ in $T^{\omega \gamma(s_k)}$.  Then $(f_{t_{k-1}\cat u})_{u\in U.L_p(X)}\subset B_{L_p(X)}$ is normally weakly null and $o(U)\geqslant \omega$ by the remark preceding the proof, so that the previous paragraph yields the existence of some $u'\in U.L_p(X)$, $g_k\in \text{co}(f_u: t_{k-1}\prec u\preceq t_{k-1}\cat u')$, and some simple function $h_k\in L_p(Z_k)$ such that $\|g_k-h_k\|_{L_p(X)}<\delta$. Let $t_k=t_{k-1}\cat u'$. In order to apply the remark before the proof, we note that since $s_{k-1}\prec s_k$, $\gamma(s_{k-1})\geqslant \gamma(s_k)+1$.  Since $$\omega \gamma(s_{k-1})\geqslant \omega(\gamma(s_k)+1) = \omega \gamma(s_k)+\omega, $$ the remark preceding the proof applies.  Note that $C_k:=\text{range}(\overline{h}_k)\subset B_{Z_k}$.   This completes the recursive construction.  Since $\Gamma_{\xi,n}.X$ is well-founded, eventually this process terminates.  The resulting $s=(\zeta_i, Z_i)_{i=1}^k\in MAX(\Gamma_{\xi,n}.X)$ clearly satisfies the conclusions.

\end{proof}

\section{Definition of an associated space and two games}

\subsection{The associated space and its properties}

If $E$ is a vector space with seminorm $\|\cdot\|$, we say a sequence $(e_i)_{i=1}^n$ in $E$ is $1$-\emph{unconditional} provided that for any scalars $(a_i)_{i=1}^n$ and any $(\ee_i)_{i=1}^n \in \{\pm 1\}^n$, $\|\sum_{i=1}^n \ee_i a_i e_i\|=\|\sum_{i=1}^n a_i e_i\|$.   Recall that for $1<p<\infty$, a vector space  $E$ with seminorm $\|\cdot\|$ which is spanned by the  $1$-unconditional basis $(e_i)_{i=1}^n$ is called $p$-\emph{concave} provided there exists a constant $C$ such that for any $(f_i)_{i=1}^n\subset L_p$, $$\|\sum_{i=1}^n f_i e_i\|_{L_p(E)}\leqslant C\|\sum_{i=1}^n \|f_i\|_{L_p}e_i\|_E.$$  The smallest such constant $C$ is denoted by $M_{(p)}(E)$.

Given $x\in \text{span}(e_i: 1\leqslant i\leqslant n)$, where $(e_i)_{i=1}^n$ is a Hamel basis for the seminormed space $E$, we write $x=\sum_{i=1}^n a_i e_i$ and $\supp(x)=\{i\leqslant n: a_i \neq 0\}$.  We say the vectors $x_1, \ldots, x_n\in \text{span}(e_i: 1\leqslant i\leqslant n)$ are \emph{disjointly supported} if the sets $\supp(x_1)$, $\ldots$, $\supp(x_n)$ are pairwise disjoint. 

For $1<\beta<\infty$, we say that an unconditional Hamel basis $(e_i)_{i=1}^n$ for a seminormed space $E$ \emph{satisfies an} $1$-\emph{lower} $\ell_\beta$ \emph{estimate} provided that for any $m\in \nn$ and any disjointly supported elements $(x_i)_{i=1}^m\subset E$, $$\bigl(\sum_{i=1}^m \|x_i\|^\beta\bigr)^{1/\beta} \leqslant  \|\sum_{i=1}^m x_i\|.$$

\begin{theorem}\cite[Theorem $1.$f$.7$]{LT} Fix $1<\beta<p<\infty$. There exists a constant $C'=C'(\beta,p)$ such that if $(e_i)_{i=1}^n$ is a $1$-unconditional basis for the seminormed space $E$ which satisfies a $1$-lower $\ell_\beta$ estimate,  then $E$ is $p$-concave and $M_{(p)}(E)\leqslant C'$.   
\label{LT}
\end{theorem}

For the remainder of this section, $T$ is a fixed, non-empty $B$-tree.   

For a non-empty set $J$, we let $c_{00}(J)$ be the span of the canonical Hamel basis $(e_j)_{j\in J}$ in the space of scalar-valued functions on $J$, where $e_j$ is the indicator of the singleton $\{j\}$.    We let $e^*_j$ denote the coordinate functional to $e_j$.    Given $x\in c_{00}(J)$, we may write $x=\sum_{j\in J}a_je_j$.  Then we define $|x|$ to be $\sum_{j\in J}|a_j|e_j$.   A \emph{suppression projection} is an operator $P$ from $\text{span}(e_j^*:j\in J)$ into itself such that there exists a subset $F$ of $J$ such that $P\sum_{j\in J}a_je_j^*= \sum_{j\in F}a_je_j^*$.

For $0<\phi< \theta<1$, let \begin{align*} N_{\theta, \phi,T}=\{0\}\cup \Bigl\{\theta \sum_{i=1}^k e^*_{t|_{j_i}}: &   \hspace{3mm}t=(\zeta_i, Z_i, C_i)_{i=1}^{|t|}\in T.X.\mathcal{C}, 1\leqslant j_1<\ldots <j_k\leqslant |t|, \\ & \hspace{3mm} \prod_{i=1}^k Z_{j_i}\cap C_{j_i}\in \mathcal{H}^K_\phi\Bigr\}\subset \text{span}(e^*_t: t\in T.X.\mathcal{C}).\end{align*}   

For $0<\phi<\theta<1$ and $1<\alpha<\infty$, let \begin{align*} M_{\theta, \phi, \alpha,T}=\Bigl\{\sum_{i=1}^k a_i g_i: &  g_i\in \cup_{n=1}^\infty N_{\theta^n, \phi^n,T}, a_i\geqslant 0, \sum_{i=1}^k a_i^\alpha\leqslant 1,  \supp(g_i) \text{\ are pairwise disjoint}\Bigr\}.\end{align*} Note that the set $M_{\theta, \phi, \alpha,T}$ is closed under suppression projections.   

We define the seminorm $\|\cdot\|_{\theta, \phi, \alpha,T}$ on $c_{00}(T.X.\mathcal{C})$ by $$\|x\|_{\theta, \phi, \alpha,T}=\sup\{f(|x|): f\in M_{\theta, \phi, \alpha, T}\}.$$

\begin{claim} Fix $1<\alpha<\infty$ and $0<\phi<\theta<1$. For any  $t\in T.X.\mathcal{C}$, $(e_{t|_i})_{i=1}^{|t|}$ is $1$-unconditional and satisfies a $1$-lower $\ell_\beta$ estimate in its span, where $1/\alpha+1/\beta=1$.   

\label{claim1}
\end{claim}

\begin{proof} Note that $1$-unconditionality is obvious.  Fix $x_1, \ldots, x_n\in \text{span}(e_{t|_i}: 1\leqslant i \leqslant |t|)$ with disjoint supports.  That is, there exist pairwise disjoint subsets $S_1, \ldots, S_n$ of $\{1, \ldots, |t|\}$ such that $x_i\in \text{span}(e_{t|_j}: j\in S_i)$.   Then there exist $g_1, \ldots, g_n\in M_{\theta, \phi, \alpha,T}$ such that for each $1\leqslant i\leqslant n$, $g_i(|x_i|)=\|x_i\|_{\theta, \phi, \alpha,T}$. Since $M_{\theta, \phi, \alpha,T}$ is closed under suppression projections, we may assume that $\supp(g_i)\subset S_i$ for each $1\leqslant i\leqslant n$.  Then if $(a_i)_{i=1}^n$ are such that $\sum_{i=1}^n a_i^\alpha=1$, $a_i\geqslant 0$, and $\sum_{i=1}^n a_i\|x_i\|_{\theta, \phi, \alpha,T}=(\sum_{i=1}^n \|x_i\|_{\theta, \phi, \alpha,T}^\beta)^{1/\beta}$, $g:=\sum_{i=1}^n a_i g_i\in M_{\theta, \phi, \alpha,T}$ and $$\|\sum_{i=1}^n x_i\|_{\theta, \phi, \alpha,T} \geqslant g\bigl(\bigl|\sum_{i=1}^n x_i\bigr|\bigr) = \sum_{i=1}^n a_ig_i(|x_i|)=\bigl(\sum_{i=1}^n \|x_i\|_{\theta, \phi, \alpha,T}^\beta)^{1/\beta}.$$

\end{proof}

\begin{claim} Fix $1<\alpha <\infty$ and $0<\phi<\theta<1$.    For any $t=(\zeta_i, Z_i, C_i)_{i=1}^k\in T.X.\mathcal{C}$, any sequence $(x_i)_{i=1}^k\in \prod_{i=1}^k Z_i\cap C_i$, and any sequence $(a_i)_{i=1}^k$ of non-negative scalars, $$\varrho(\sum_{i=1}^k a_ix_i)\leqslant \frac{1}{\theta-\phi}\|\sum_{i=1}^k a_i e_{t|_i}\|_{\theta, \phi, \alpha,T}.$$ 
\label{claim2}
\end{claim}

\begin{proof} We recall that if $C\in \mathcal{C}$, $C\subset B_X$ by the definition of $\mathcal{C}$. With $t$, $(x_i)_{i=1}^k\in \prod_{i=1}^k Z_i\cap C_i$, and $(a_i)_{i=1}^k$ as in the statement, fix $x^*\in K$ such that $\text{Re\ }x^*(\sum_{i=1}^k a_ix_i)=\varrho(\sum_{i=1}^k a_ix_i)$.  For all $j\in \nn$, let $B_j=\{i\leqslant k: \text{Re\ }x^*(x_i)\in (\phi^j, \phi^{j-1}]\}$. Note that for every $j\in \nn$, $\theta^j\sum_{i\in B_j} e^*_{t|_i}\in N_{\theta^j, \phi^j, T}$, so $$\phi^{j-1}\sum_{i\in B_j}a_i=\phi^{-1}(\phi/\theta)^j (\theta^j\sum_{i\in B_j} e^*_{t|_i})(\sum_{i=1}^k a_i e_{t|_i})\leqslant \phi^{-1}(\phi/\theta)^j \|\sum_{i=1}^k a_i e_{t|_i}\|_{\theta, \phi, \alpha,T}.$$    Then \begin{align*} \varrho(\sum_{i=1}^k a_ix_i) & \leqslant \sum_{j=1}^\infty \sum_{i\in B_j} a_i\text{Re\ }x^*(x_i)  \leqslant \sum_{j=1}^\infty \phi^{j-1}\sum_{i\in B_j} a_i \leqslant \sum_{j=1}^\infty \phi^{-1}(\phi/\theta)^j \|\sum_{i=1}^k a_ie_{t|_i}\|_{\theta, \phi, \alpha,T}\\ & =\frac{1}{\theta-\phi}\|\sum_{i=1}^k a_ie_{t|_i}\|_{\theta, \phi, \alpha,T}.\end{align*}

\end{proof}

\begin{corollary} Fix $1<p, \alpha, \beta<\infty$ with $1/\alpha+1/\beta=1$ and $\beta<p$.  Let $C'=C'(\beta, p)$ be the constant from Theorem \ref{LT}.  Suppose that $\xi$ is an ordinal, $n$ is a natural number, $\ee>0$, and $0<\phi<\theta<1$ are such that Player I has a winning strategy in the game with target set $$\Bigl\{t\in MAX(\Gamma_{\xi,n}.X.\mathcal{C}): \|\sum_{s\preceq t}\mathbb{P}_{\xi,n}(s) e_s\|_{\theta, \phi, \alpha, \Gamma_{\xi,n}}>\ee\Bigr\}.$$ Then for any $B$-tree $T$ with $o(T)\geqslant \omega^{1+\xi}n$ and any normally weakly null $(f_t)_{t\in T.L_p(X)}\subset B_{L_p(X)}$, $$\inf\bigl\{ \|\varrho(f)\|_{L_p}: t\in T.L_p(X), f\in \text{\emph{co}}(f_s: \varnothing\prec s\preceq t)\bigr\} \leqslant \frac{C' \ee}{n(\theta-\phi)}.$$  

\label{espoo}
\end{corollary}

\begin{proof} Recall for the proof that for a simple function $h\in L_p(X)$, $\overline{h}$ is the function in $L_p(X)$ such that $\overline{h}(\varpi)=0$ if $h(\varpi)=0$ and $\overline{h}(\varpi)=h(\varpi)/\|h(\varpi)\|$ otherwise.  

Fix a winning strategy $\psi$ for Player I in the game with the indicated target set. Fix $\delta>0$.   By Proposition \ref{prop8}, there exist $s=(\zeta_i , Z_i)_{i=1}^k\in MAX(\Gamma_{\xi,n}.X)$, $\varnothing=t_0\prec \ldots \prec t_k$, $g_i\in \text{co}(f_u:t_{i-1}\prec u\preceq t_i)$,  simple functions $h_i\in B_{L_p(X)}$, and $C_i\in \mathcal{C}$ such that $\|g_i-h_i\|_{L_p(X)}<\delta$, $\text{range}(\overline{h}_i)= C_i\subset B_{Z_i}$, and  $(\zeta_i, Z_i)= \psi((\zeta_j, Z_j, C_j)_{j=1}^{i-1})$.  This means that for any $\varpi\in [0,1]$, $(\overline{h}_i(\varpi))_{i=1}^k\in \prod_{i=1}^k Z_i\cap C_i$, whence by Claim \ref{claim2}, for any non-negative scalars $(a_i)_{i=1}^k$, $$\varrho(\sum_{i=1}^k a_ih_i(\varpi)) = \varrho(\sum_{i=1}^k a_i\|h_i(\varpi)\|\overline{h}_i(\varpi))\leqslant \frac{1}{\theta-\phi}\|\sum_{i=1}^k a_i \|h_i(\varpi)\| e_{s|_i}\|_{\theta, \phi, \alpha, \Gamma_{\xi,n}}.$$ Since by Claim \ref{claim2} $(e_u)_{u\preceq s}$ satisfies a lower $\ell_\beta$ estimate in its span,  we deduce that \begin{align*} \|\varrho(\sum_{i=1}^k n^{-1}\mathbb{P}_{\xi,n}(s|_i)h_i)\|_{L_p} & = \Bigl( \int_0^1 \Bigl|\varrho(\sum_{i=1}^k n^{-1}\mathbb{P}_{\xi,n}(s|_i)\|h_i(\varpi)\|\overline{h}_i(\varpi))\Bigr|^p d\varpi\Bigr)^{1/p}  \\ &  \leqslant \frac{1}{n(\theta-\phi)} \Bigl(\int_0^1 \Bigl\|\sum_{u\preceq s}  \mathbb{P}_{\xi,n}(u)\|h_{|u|}(\varpi)\| e_u\Bigr\|_{\theta, \phi, \alpha, \Gamma_{\xi,n}}^pd\varpi\Bigr)^{1/p} \\ & \leqslant \frac{C'}{n(\theta-\phi)} \Bigl\|\sum_{u\preceq s}\mathbb{P}_{\xi,n}(u) \|h_{|u|}\|_{L_p(X)}e_u\Bigr\|_{\theta, \phi, \alpha, \Gamma_{\xi,n}} \\ & \leqslant  \frac{C'\ee}{n(\theta-\phi)}.\end{align*}

Here we have used $1$-unconditionality, $\|h_i\|_{L_p(X)}\leqslant 1$ for each $1\leqslant i\leqslant k$, and the fact that since $\psi$ is a winning strategy for Player I, $$\|\sum_{u\preceq s}\mathbb{P}_{\xi,n}(u) \|h_{|u|}\|_{L_p(X)} e_u \|_{\theta, \phi, \alpha, \Gamma_{\xi,n}}\leqslant \|\sum_{u\preceq s} \mathbb{P}_{\xi,n}(u)e_u\|_{\theta, \phi, \alpha, \Gamma_{\xi,n}}\leqslant \ee.$$ Let $g=n^{-1}\sum_{u\preceq s}\mathbb{P}_{\xi,n}(u) g_i\in \text{co}(f_u:u\preceq t_k)$ and $h=n^{-1}\sum_{u\preceq s} \mathbb{P}_{\xi,n}(u) h_i$.   Since $\varrho$ is $1$-Lipschitz, it follows that $\|\varrho(g)-\varrho(g)\|_{L_p}\leqslant \|g-h\|_{L_p(X)}<\delta$, so that $$\|\varrho(g)\|_{L_p}\leqslant \delta+\|\varrho(h)\|_{L_p}\leqslant \delta+\frac{C'\ee}{n(\theta-\phi)}.$$ Since $\delta>0$ was arbitrary, we are done.

\end{proof}

\subsection{Particular games on $\Gamma_{\xi,n}.X.\mathcal{C}$}

The statement of Proposition \ref{technical} is notationally cumbersome.  We isolate the following result as a way of using Proposition \ref{technical}.  

\begin{lemma} Fix  $0<\phi<\theta <1$.  Suppose  that $\xi$ is an ordinal, $m,n$ are natural numbers, $(C_s)_{s\in \Gamma_{\xi,n}.X}\subset \mathcal{C}$, and $(\sigma, \tau):\Gamma_{\xi,m}.X\to \Gamma_{\xi,n}.X$ is an extended pruning. For $t=(\zeta_i, Z_i)_{i=1}^k\in \Gamma_{\xi,n}.X$, let $r(t)=(\zeta_i, Z_i, C_{t|_i})_{i=1}^k$.  If $\nu\in \nn$ is such that for every $s\in MAX(\Gamma_{\xi,m}.X)$, there exists a functional $h_s\in \cup_{l=1}^\nu N_{\theta^i, \phi^i, \Gamma_{\xi,n}}$ such that $\cup_{t\preceq s} r(\sigma(t))\subset \text{\emph{supp}}(h_s)$, then $Sz(K, \phi^\nu/2)>\omega^\xi m$.   
\label{criteria}
\end{lemma}

\begin{proof} For $s=(\zeta_i, Z_i)_{i=1}^k\in \Gamma_{\xi,n}.X$, let $\lambda(s)=Z_k\cap C_s$. For $s\in \Gamma_{\xi,m}.X$, let $\mathfrak{s}(s)=\prod_{i=1}^{|s|} \lambda(\sigma(s|_i))$.   

Fix $s\in MAX(\Gamma_{\xi,m}.X)$ and let $h_s\in \cup_{l=1}^\nu N_{\theta^l, \phi^l, \Gamma_{\xi,n}}$ be as in the statement of the lemma and fix $1\leqslant l\leqslant \nu$ such that $h_s\in N_{\theta^l, \phi^l, \Gamma_{\xi,n}}$.   We will prove that $\mathfrak{s}(s)\in \mathcal{H}^K_{\phi^\nu}$.   Since for any $1\leqslant m\leqslant k$ and any $C_1', \ldots, C_k'\in  \mathcal{C}$ such that  $\prod_{i=1}^k C_i'\in \mathcal{H}^K_{\phi^\nu}$,   $\prod_{i=1}^m C_i' \in \mathcal{H}^K_{\phi^\nu}$, this will show that for any non-empty initial segment $s_1$ of $s$, $\mathfrak{s}(s_1)\in \mathcal{H}^K_{\phi^\nu}$.  From here, an appeal to Proposition \ref{technical} will finish the proof.   

Fix $u=(\mu_i, W_i, C_i)_{i=1}^{|u|}\in \Gamma_{\xi,n}.X.\mathcal{C}$ and $1\leqslant j_1<\ldots <j_\mu\leqslant |u|$ such that $h_s=\theta^l\sum_{i=1}^\mu e^*_{u|_{j_i}}$ and $\prod_{i=1}^\mu W_{j_i}\cap C_{j_i}\in \mathcal{H}^K_{\phi^l}$.   Let $\tau(s)=t=(\zeta_i, Z_i)_{i=1}^\eta$.  For each $1\leqslant i\leqslant |s|$, let $l_i=|\sigma(s|_i)|$.  Note that for all $1\leqslant i\leqslant |s|$, $r(\sigma(s|_i))=(\zeta_j, Z_j, C_{t|_j})_{j=1}^{l_i}$ and $\mathfrak{s}(s)=\prod_{j=1}^{|s|} Z_{l_j}\cap C_{t|_{l_j}}.$   By hypothesis, \begin{align*} \{(\zeta_j, Z_j, C_{t|_j})_{j=1}^{l_i}: 1\leqslant i\leqslant |s|\} & = \{r(\sigma(s|_i)):1\leqslant i\leqslant |s|\} \\ & \subset \text{supp}(h_s) = \{(\mu_j, W_j, C_j)_{j=1}^{j_i}: 1\leqslant i\leqslant \mu\}.\end{align*} From this it follows that there exist $m_1<\ldots <m_{|s|}$ such that for every $1\leqslant i\leqslant |s|$, $r(\sigma(s|_i))=u|_{j_{m_i}}$.  Choose $(x_i)_{i=1}^\mu\in \prod_{i=1}^\mu W_{j_i}\cap C_{j_i}$ such that there exists $x^*\in K$ so that $\text{Re\ }x^*(x_i)\geqslant \phi^l$ for each $1\leqslant i\leqslant \mu$, which exists because $\prod_{i=1}^\mu W_{j_i}\cap C_{j_i}\in \mathcal{H}^K_{\phi^l}$. Since $Z_{l_i}=W_{j_{m_i}}$ and $C_{t|_{l_i}}=C_{j_{m_i}}$, $(x_{m_i})_{i=1}^{|s|}\in \prod_{i=1}^{|s|} Z_{l_i}\cap C_{t|_{l_i}},$ which shows that    $\mathfrak{s}(s)\in \mathcal{H}^K_{\phi^l}. $  Since $l\leqslant \nu$, $\mathcal{H}^K_{\phi^l}\subset \mathcal{H}^K_{\phi^\nu}$, so that $\mathfrak{s}(s)\in \mathcal{H}^K_{\phi^\nu}$.

\end{proof}

\begin{lemma} Fix $1<\alpha<\infty$ and $0<\phi<\theta<1$.  If $Sz(K)\leqslant \omega^\xi$, then for any $\ee>0$, Player I has a winning strategy in the game with target set $$\Bigl\{t\in MAX(\Gamma_\xi.X.\mathcal{C}): \|\sum_{s\preceq t}\mathbb{P}_\xi(s) e_s\|_{\theta, \phi, \alpha, \Gamma_\xi} > \ee\Bigr\}.$$

\label{game lemma}
\end{lemma}

\begin{proof} Suppose not. Then by Proposition \ref{games}, there exist $\ee>0$ and  $(C_s)_{s\in \Gamma_\xi.X}\subset \mathcal{C}$ such that  $$\ee< \inf\Bigl\{\|\sum_{s\preceq t} \mathbb{P}_\xi(s) e_{r(s)}\|_{\theta, \phi, \alpha, \Gamma_\xi}: t\in MAX(\Gamma_\xi.X)\Bigr\}.$$ For $s=(\zeta_i, Z_i)_{i=1}^k\in \Gamma_\xi.X$, let $r(s)=(\zeta_i, Z_i, C_{s|_i})_{i=1}^k$.   For every $t\in MAX(\Gamma_\xi.X)$, fix $f_t\in M_{\theta, \phi, \alpha, \Gamma_\xi}$ such that $\supp(f_t)\subset [\preceq r(t)]$ and  $f_t(\sum_{s\preceq t}\mathbb{P}_\xi(s) e_{r(s)})=\|\sum_{s\preceq t} \mathbb{P}_\xi(s)e_{r(s)}\|_{\theta, \phi,\alpha, \Gamma_\xi}$.  Define $F:\Pi(\Gamma_\xi.X)\to \rr$ by letting $F(s,t)=f_t(e_{r(s)})$.   By Theorem \ref{stabilization}, there exists an extended pruning $(\sigma, \tau):\Gamma_\xi.X\to \Gamma_\xi.X$ such that $$\ee<\inf_{(s,t)\in \Pi(\Gamma_\xi.X)} F(\sigma(s), \tau(t)).$$   Fix $\nu\in \nn$ such that $\ee>\theta^\nu$ and for each $t\in MAX(\Gamma_\xi.X)$, write $f_{\tau(t)}=\sum_{i=1}^{k_t} a_{i,t}g_{i,t}$ where $a_{i,t}\geqslant 0$, $\sum_{i=1}^{k_t} a_{i,t}^\alpha\leqslant 1$, and $g_{i,t}\in \cup_{n=1}^\infty N_{\theta^n, \phi^n, \Gamma_\xi}$ have pairwise disjoint supports.    For each $t\in MAX(\Gamma_\xi.X)$, let $$R_t=\{i\leqslant k_t: a_{i,t}\geqslant \ee\}.$$ Since $\sum_{i=1}^{k_t} a_{i,t}^\alpha\leqslant 1$, $|R_t|\leqslant \lfloor 1/\ee^\alpha\rfloor =:k_0$.   Note that since $\ee<f_{\tau(t)}(e_{r(\sigma(s))})$ for any $\varnothing \prec s\preceq t$, $r(\sigma(s))\in \cup_{i\in R_t} \supp(g_{i,t})$.   We write $\sum_{i\in R_t} a_{i,t}g_{i,t}=\sum_{i=1}^{l_t} b_{i,t} h_{i,t}$ where $l_t\leqslant k_0$, $(b_{i,t})_{i=1}^{l_t}$ is an enumeration of $(a_{i,t})_{i\in R_t}$, and $(h_{i,t})_{i=1}^{l_t}$ is the corresponding enumeration of $(g_{i,t})_{i\in R_t}$.    Define $\kappa:\Pi(\Gamma_\xi.X)\to \{1, \ldots, k_0\}$ by letting $\kappa(\sigma, \tau)$ be the unique $i\leqslant l_t$ such that $r(\sigma(s))\in \text{supp}(h_{i,t})$.   By Theorem \ref{stabilization}$(ii)$, there exists an extended pruning $(\sigma',\tau'):\Gamma_\xi.X\to \Gamma_\xi.X$ and $1\leqslant l\leqslant k_0$ such that $\kappa(\sigma'(s), \tau'(t))=l$ for all $(s, t)\in \Pi(\Gamma_\xi.X)$.    We now note that for any $s\in MAX(\Gamma_\xi.X)$, $h_{l, \tau'(s)}\in \cup_{i=1}^\nu N_{\theta^i, \phi^i, \Gamma_\xi}$ is such that $$\{r(\sigma\circ \sigma'(u)): \varnothing\prec u\preceq s\}\subset \text{supp}(h_{l, \tau'(s)}),$$ and an appeal to Lemma \ref{criteria} yields that $Sz(K, \phi^\nu/2)>\omega^\xi$.  This contradiction finishes the proof.  To see that $h_{l, \tau'(s)}\in \cup_{i=1}^\nu N_{\theta^i, \phi^i, \Gamma_\xi}$, we note that if $h_{l, \tau'(s)}\in N_{\theta^i, \phi^i, \Gamma_\xi}$, $$\ee\leqslant h_{l, \tau'(t)}(e_{r(\sigma\circ \sigma'(s))}) \leqslant \|h_{l, \tau'(t)}\|_\infty \leqslant \theta^i.$$   This shows that $i\leqslant \nu$ by our choice of $\nu$.

\end{proof}

\begin{lemma} Fix  $1<\alpha, \beta<\infty$ and $0<\phi<2^{-1/\alpha}$ and assume that $1/\alpha+1/\beta=1$.  Assume that for some $C\geqslant 1$ and all $i\in \nn$, $Sz_\xi(K, \phi^i/2)\leqslant C 2^i$.  Let $\theta=2^{-1/\alpha}$.   Then for any $n\in \nn$ and any $C_1>C$, Player I has a winning strategy in the game with target set $$\Bigl\{t\in MAX(\Gamma_{\xi,n}.X.\mathcal{C}): \|\sum_{s\preceq t}\mathbb{P}_{\xi,n}(s) e_s\|_{\theta, \phi, \alpha, \Gamma_{\xi,n}}> C_1n^{1/\beta}\Bigr\}.$$

\label{game lemma 2}
\end{lemma}

\begin{proof} Suppose not.  Then for some $n\in \nn$, there exist $(C_s)_{s\in \Gamma_{\xi,n}.X}\subset \mathcal{C}$ and $$(f_t)_{t\in MAX(\Gamma_{\xi,n}.X)}\subset M_{\theta, \phi, \alpha, \Gamma_{\xi,n}}$$ such that $$C n^{1/\beta} < \inf_{t\in MAX(\Gamma_{\xi,n}.X)} f_t\bigl(\sum_{s\preceq t}\mathbb{P}_{\xi,n}(s) e_{r(s)}\bigr).$$  We may assume as in Lemma \ref{game lemma} that $\supp(f_t)\subset [\preceq r(t)]$ for each $t\in MAX(\Gamma_{\xi,n}.X)$.    Then by Theorem \ref{stabilization}$(i)$, there exist a level preserving extended pruning $(\sigma, \tau):\Gamma_{\xi,n}\to \Gamma_{\xi,n}$ and numbers, $b_1, \ldots, b_n$ such that $C  n^{1/\beta}<\sum_{i=1}^n b_i$ and for all $1\leqslant i\leqslant n$ and all $\Lambda_{\xi,n,i}\ni s\preceq t\in MAX(\Gamma_{\xi,n})$, $f_{\tau(t)}(e_{r(\sigma(s))})\geqslant b_i$.   Fix $\delta>0$ such that $Cn^{1/\beta}+n\delta< \sum_{i=1}^n b_i$.  Let $R=\{i\leqslant n: b_i\geqslant \delta\}$.   

\begin{sublemma} There exist a level preserving extended pruning $(\sigma_0, \tau_0):\Gamma_{\xi,n}.X\to \Gamma_{\xi,n}.X$, $l,w\in \nn$, $(a_i)_{i=1}^l\in B_{\ell_\alpha^l}$, $(k_i)_{i\in R}\subset \{1, \ldots, l\}$, $(w_i)_{i=1}^l\subset \{1, \ldots, w\}$, and $(g_t)_{t\in MAX(\Gamma_{\xi,n}.X)}\subset M_{\theta,\phi, \alpha, \Gamma_{\xi,n}}$ such that \begin{enumerate}[(i)]\item for each $t\in MAX(\Gamma_{\xi,n}.X)$, $\|g_t-f_{\tau\circ \tau_0(t)}\|_\infty<\delta$, \item for any $t\in MAX(\Gamma_{\xi,n}.X)$, there exist disjointly supported functionals $h_{1,t}, \ldots, h_{l,t}$ such that $h_{i,t}\in N_{\theta^{w_i}, \phi^{w_i}, \Gamma_{\xi,n}}$ and $g_t=\sum_{i=1}^l a_ih_{i,t}$, \item for $i\in R$ and $\Lambda_{\xi,n,i}\ni s\preceq t\in MAX(\Gamma_{\xi,n}.X)$, $r(\sigma\circ \sigma_0(s))\in \text{\emph{supp}}(h_{k_i, t})$,  \end{enumerate}

\end{sublemma}

We first finish the proof of the lemma and then return to the proof of the sublemma.  Note that item $(iii)$ of the sublemma implies that for $i\in R$ and $\Lambda_{\xi,n,i}\ni s\preceq t\in MAX(\Gamma_{\xi,n})$, $$b_i\leqslant g_t(e_{r(\sigma\circ \sigma_0(s))})+\delta= a_{k_i}\theta^{w_{k_i}}+\delta.$$  From this and our choice of $\delta$ we deduce that $$Cn^{1/\beta}+\delta n < \sum_{i=1}^n b_i \leqslant \delta n +\sum_{i\in R} a_{k_i}\theta^{w_{k_i}}.$$  Partition $R$ into sets $R_1, \ldots, R_l$, where $R_j=\{ i\in R: k_i=j\}$, so that $$C n^{1/\beta} < \sum_{i\in R} a_{k_i} \theta^{w_{k_i}} = \sum_{j=1}^l a_j \theta^{w_j}|R_j|.$$  We  claim that for each $j$, $|R_j|\leqslant C2^{w_j}$.  Indeed, suppose $|R_j|>C2^{w_j}$ for some $j$.  By Theorem \ref{stabilization}$(iv)$, if $R_j=\{r_1, \ldots, r_m\}$, with $r_1<\ldots <r_m$, there exists extended pruning $(\sigma', \tau'):\Gamma_{\xi,m}.X\to \Gamma_{\xi, n}.X$ such that $\sigma'(\Lambda_{\xi, m,i})\subset \Lambda_{\xi,n,r_i}$.   We now use Lemma \ref{criteria} to deduce that $Sz_\xi(K, \phi^{w_j}/2)>C2^{w_j}$, which is a contradiction.  Thus we deduce that $|R_j|\leqslant C 2^{w_j}$ for each $j$.   This means that for each $1\leqslant j\leqslant l$, $$\theta^{w_j} = (2^{-1/\alpha})^{w_j}=(2^{w_j})^{-1/\alpha} \leqslant C^{1/\alpha} |R_j|^{-1/\alpha}\leqslant C |R_j|^{-1/\alpha}.$$

Then \begin{align*} \sum_{j=1}^l a_j \theta^{w_j}|R_j| & \leqslant C\sum_{j=1}^l a_j|R_j|^{1-1/\alpha} = C\sum_{j=1}^l a_j|R_j|^{1/\beta}\leqslant C \bigl(\sum_{j=1}^l |a_j|^\alpha\bigr)^{1/\alpha}\bigl(\sum_{j=1}^n |R_j|\bigr)^{1/\beta} \\ & \leqslant C|R|^{1/\beta} \leqslant Cn^{1/\beta}.\end{align*} Thus we reach a contradiction.  

We now return to the proof of the sublemma. First fix $w\in \nn$ such that $\theta^w<\delta$.   For each $t\in MAX(\Gamma_{\xi,n})$, write $f_{\tau(t)}=\sum_{i=1}^{k_t} a_{i,t}f_{i,t}$ for some disjointly supported $f_{i,t}\in \cup_{j=1}^\infty N_{\theta^j, \phi^j, \Gamma_{\xi,n}}$ and $a_{i,t}\geqslant 0$ such that $\sum_{i=1}^{k_t}a_{i,t}^\alpha\leqslant 1$.    Let $S_t=\{i\leqslant k_t: \|a_{i,t}f_{i,t}\|_\infty\geqslant \delta\}$.  Note that since $\sum_{i=1}^{k_t}a_{i,t}^\alpha\leqslant 1$, $|S_t|\leqslant \lfloor 1/\delta^\alpha\rfloor=:k_0$.   As in the previous lemma, we write $\sum_{i\in S_t}a_{i,t}f_{i,t}=\sum_{i=1}^{l_t} a'_{i,t}f_{i,t}'$ for some $l_t\leqslant k_0$.    Considering the function from $MAX(\Gamma_{\xi,n}.X)$ given by $t\mapsto l_t\in \{1, \ldots, k_0\}$, we use Theorem \ref{stabilization}$(iii)$ to obtain $l\in \nn$ and a level preserving extended pruning $(\sigma', \tau'):\Gamma_{\xi,n}.X\to \Gamma_{\xi,n}.X$ such that for all $t\in MAX(\Gamma_{\xi,n}.X)$, $l_{\tau'(t)}=l$.  Note that since $ \|a'_{i, \tau'(t)}f'_{i, \tau'(t)}\|_\infty \geqslant \delta$ for every $1\leqslant i\leqslant l$ and $t\in MAX(\Gamma_{\xi,n})$, if $f'_{i, \tau'(t)}\in N_{\theta^j, \phi^j, \Gamma_{\xi,n}}$, $j\leqslant w$.   Let $w_{i, \tau'(t)}$ be the value $j\in \{1, \ldots, w\}$ such that $f'_{i, \tau'(t)}\in N_{\theta^j, \phi^j, \Gamma_{\xi,n}}$.  By considering the map from $MAX(\Gamma_{\xi,n}.X)$ into $B_{\ell_\alpha^l}\times \{1, \ldots, w\}^l$ given by $$t\mapsto \bigl((a_{i, \tau'(t)})_{i=1}^l, (w_{i, \tau'(t)})_{i=1}^l\bigr),$$  we use Theorem \ref{stabilization}$(iii)$ again to find another level preserving extended pruning $(\sigma'', \tau''):\Gamma_{\xi,n}.X\to \Gamma_{\xi,n}.X$, $(a_i)_{i=1}^l\in B_{\ell_\alpha^l}$ and $(w_i)_{i=1}^l \subset \{1, \ldots, w\}$ such that for all $t\in MAX(\Gamma_{\xi,n})$, $\|(a_{i, \tau'\circ \tau''(t)})_{i=1}^l-(a_i)_{i=1}^l\|_{\ell_\alpha^l}<\delta$ and for all $1\leqslant i\leqslant l$, $f_{i, \tau'\circ \tau''(t)}'\in N_{\theta^{w_i}, \phi^{w_i}, \Gamma_{\xi,n}}$.   Note that for all $t\in MAX(\Gamma_{\xi,n}.X)$, $$\|f_{\tau\circ \tau'\circ \tau''(t)}-\sum_{i=1}^l a_i f'_{i, \tau'\circ \tau''(t)}\|_\infty <\delta.$$  This implies that for any $i\in R$ and any $\Lambda_{\xi,n, i}\ni s\preceq t$, since $$\delta\leqslant b_i \leqslant f_{\tau\circ\tau'\circ \tau''}(e_{r(\sigma\circ \sigma'\circ \sigma''(s))}),$$ $r(\sigma\circ \sigma'\circ \sigma''(s))\in \cup_{j=1}^l \text{supp}(f'_{j, \tau'\circ\tau''(t)})$.  Thus we may let $\kappa(s,t)$ be the unique $j\in \{1, \ldots, l\}$ such that $r(\sigma\circ \sigma'\circ \sigma''(s))\in \text{supp}(f_{j, \tau'\circ \tau''(t)})$ if $s\in \cup_{i\in R}\Lambda_{\xi,n,i}$, and $\kappa(s,t)=0$ otherwise.    Applying Theorem \ref{stabilization}$(ii)$, we deduce the existence of $(k_i)_{i\in R}\subset \{1, \ldots, l\}$ and a level preserving extended pruning $(\sigma''', \tau'''):\Gamma_{\xi,n}.X\to \Gamma_{\xi,n}.X$ such that setting  $\sigma_0=\sigma'\circ \sigma''\circ \sigma'''$, $\tau_0=\tau'\circ \tau''\circ \tau'''$, $h_{i,t}=f'_{i, \tau_0(t)}$, and $g_t=\sum_{i=1}^l a_i h_{i,t}$, finishes the proof.

\end{proof}

\section{Proof of the main results}

\begin{proof}[Proof of Theorem \ref{main theorem}] Let $\phi=1/3$ and $\theta=2/3$, so that $\frac{1}{\theta-\phi}=3$.  Fix $1<p<\infty$.  Fix any $1<\alpha, \beta<\infty$ such that $\beta<p$ and $1/\alpha+1/\beta=1$.   Let $C'=C'(\beta, p)$ be the constant from Theorem \ref{LT}.  Fix $\ee>0$.  By Lemma \ref{game lemma}, Player I has a winning strategy in the game with target set $$\Bigl\{t\in MAX(\Gamma_\xi.X.\mathcal{C}): \|\sum_{s\preceq t}\mathbb{P}_\xi(s) e_s\|_{\theta, \phi, \alpha, \Gamma_\xi} > \ee\Bigr\}.$$ By Corollary \ref{espoo}, for any $B$-tree $T$ with $o(T)=\omega^{1+\xi}$ and any normally weakly null collection $(f_t)_{t\in T.L_p(X)}\subset B_{L_p(X)}$, $$\inf\Bigl\{\|\varrho(f)\|_{L_p}: t\in T.L_p(X), f\in \text{co}(f_s:\varnothing\prec s\preceq t)\Bigr\} \leqslant 3C'\ee.$$   We deduce $Sz(K_p)\leqslant \omega^{1+\xi}$ by Theorem \ref{characterization}$(i)$.

It is clear that $Sz(K)\leqslant Sz(K_p)$ for any $1<p<\infty$.  If $K$ is convex, then either $Sz(K)=\infty$, in which case $Sz(K_p)=\infty=\omega \infty=Sz(K)$, or there exists an ordinal $\xi$ such that $Sz(K)=\omega^\xi$ \cite[Proposition $4.2$]{Calt}.  We deduce that $Sz(K_p)\leqslant \omega^{1+\xi}=\omega Sz(K)$ by the previous paragraph. In the case that $\xi\geqslant \omega$, $1+\xi=\xi$. 

\end{proof}

\begin{proof}[Proof of Theorem \ref{power type}] If $\textbf{p}_\xi(K)=\infty$, there is nothing to show, so assume $\textbf{p}_\xi(K)<\infty$. Fix $1<p,q<\infty$ with $1/p+1/q=1$.   Fix $1<\alpha, \beta, \gamma<\infty$ such that $\max\{\textbf{p}_\xi(K),q\}<\gamma< \alpha$ and $1/\alpha+1/\beta=1$. Let $C'=C'(\beta, p)$ be the constant from Theorem \ref{LT}.  Let $\phi=2^{-1/\gamma}$ and note that $\sup_{i\in \nn}\ee^\gamma Sz_\xi(K, \phi^i/2)/2^i<\infty$.  By Lemma \ref{game lemma 2}, with $\theta=2^{-1/\alpha}$, there exists a constant $C_1$ such that for every $n\in \nn$, Player I has a winning strategy in the game with target set $$\Bigl\{t\in MAX(\Gamma_{\xi,n}.X.\mathcal{C}): \|\sum_{s\preceq t}\mathbb{P}_{\xi,n}(s) e_s\|_{\theta, \phi, \alpha, \Gamma_{\xi,n}}> C_1/n^{1/\beta}\Bigr\}.$$     By Corollary \ref{espoo}, for every $n\in \nn$, every $B$-tree $T$ with $o(T)=\omega^{1+\xi}n$, and every normally weakly null $(f_t)_{t\in T.L_p(X)}\subset B_{L_p(X)}$, $$\inf\Bigl\{ \|\varrho(f)\|_{L_p}: t\in T.L_p(X), f\in \text{co}(\varnothing\prec s\preceq t)\Bigr\} \leqslant  \frac{C_1C'}{n(\theta-\phi)} n^{1/\beta}=  \frac{C_1C'}{n^{1/\alpha}\theta-\phi}. $$ By Theorem \ref{characterization}$(ii)$, $\textbf{p}_{1+\xi}(K_p)\leqslant \alpha$.  Since $\alpha>\max\{\textbf{p}_\xi(K),q\}$ was arbitrary, we deduce that $\textbf{p}_{1+\xi}(K_p)\leqslant \max\{\textbf{p}_\xi(K),q\}$.

\end{proof}

\end{document}